\documentclass[12pt]{amsart}
\usepackage{amsfonts,amssymb,amscd,amsmath,enumerate,verbatim}
\usepackage[latin1]{inputenc}
\usepackage{amscd}
\usepackage{latexsym}
\usepackage{mathptmx}
\usepackage{multicol}
\input xy
\xyoption{all}
\usepackage{hyperref}
\usepackage{color}
\hypersetup{colorlinks,linkcolor=blue ,citecolor=blue, urlcolor=blue}
\def\NZQ{\mathbb}               % the font for N,Z,Q,R,C

\def\ZZ{{\NZQ Z}}
\def\RR{{\NZQ R}}

\def\frk{\mathfrak}               % font for "Fraktur"

\def\Phi{{\frk N}}
\def\ab{{\bold a}}

\def\eb{{\bold e}}
\def\tb{{\bold t}}

\def\xb{{\bold x}}
\def\yb{{\bold y}}

\def\opn#1#2{\def#1{\operatorname{#2}}} % to make operators
\opn\chara{char} 
\opn\length{\ell} 
\opn\pd{pd} 
\opn\rk{rk}
\opn\projdim{proj\,dim} 
\opn\injdim{inj\,dim} 
\opn\rank{rank}
\opn\depth{depth} 
\opn\grade{grade} 
\opn\height{height}
\opn\embdim{emb\,dim} 
\opn\codim{codim}

\opn\Tr{Tr} 
\opn\bigrank{big\,rank}
\opn\superheight{superheight}
\opn\lcm{lcm}
\opn\trdeg{tr\,deg}%\emph{
\opn\reg{reg} 
\opn\lreg{lreg} 
\opn\ini{in} 
\opn\lpd{lpd}
\opn\size{size}
\opn\mult{mult}
\opn\dist{dist}
\opn\cone{cone}
\opn\lex{lex}
\opn\rev{rev}
\opn\bipyr{bipyr}
\opn\div{div} \opn\Div{Div} \opn\cl{cl} \opn\Cl{Cl}
\opn\Spec{Spec} \opn\Supp{Supp} \opn\supp{supp} \opn\Sing{Sing}
\opn\Ass{Ass} \opn\Min{Min}
\opn\Ann{Ann} \opn\Rad{Rad} \opn\Soc{Soc}
\opn\Syz{Syz} \opn\Im{Im} \opn\Ker{Ker} \opn\Coker{Coker}
\opn\Am{Am} \opn\Hom{Hom} \opn\Tor{Tor} \opn\Ext{Ext}
\opn\End{End} \opn\Aut{Aut} \opn\id{id} \opn\ini{in}

\opn\nat{nat}
\opn\pff{pf}%   \pf exists already
\opn\Pf{Pf} \opn\GL{GL} \opn\SL{SL} \opn\mod{mod} \opn\ord{ord}
\opn\Gin{Gin}
\opn\Hilb{Hilb}\opn\adeg{adeg}\opn\std{std}\opn\ip{infpt}
\opn\Pol{Pol}
\opn\sat{sat}
\opn\Var{Var}
\opn\Gen{Gen}

\opn\aff{aff} \opn\con{conv} \opn\relint{relint} \opn\st{st}
\opn\lk{lk} \opn\cn{cn} \opn\core{core} \opn\vol{vol}
\opn\link{link} \opn\star{star}
\opn\gr{gr}

\def\Fc{{\mathcal F}}

\def\Pc{{\mathcal P}}
\def\Qc{{\mathcal Q}}

\def\vol{{\textnormal{vol}}}
\def\conv{{\textnormal{conv}}}

\def\ord{{\textnormal{ord}}}
\def\pot#1#2{#1[\kern-0.28ex[#2]\kern-0.28ex]}

\opn\dirlim{\underrightarrow{\lim}}
\opn\inivlim{\underleftarrow{\lim}}
\def\Implies{\ifmmode\Longrightarrow \else
	\unskip${}\Longrightarrow{}$\ignorespaces\fi}
\def\implies{\ifmmode\Rightarrow \else
	\unskip${}\Rightarrow{}$\ignorespaces\fi}
\def\iff{\ifmmode\Longleftrightarrow \else
	\unskip${}\Longleftrightarrow{}$\ignorespaces\fi}

\let\:=\colon
\newtheorem{Theorem}{Theorem}[section]
\newtheorem{Lemma}[Theorem]{Lemma}
\newtheorem{Corollary}[Theorem]{Corollary}
\newtheorem{Proposition}[Theorem]{Proposition}

\newtheorem{Example}[Theorem]{Example}

\newtheorem{Question}[Theorem]{Question}
\numberwithin{equation}{section}

\let\epsilon\varepsilon
\let\phi=\varphi
\let\kappa=\varkappa
\def\qed{\ifhmode\textqed\fi
	\ifmmode\ifinner\quad\qedsymbol\else\dispqed\fi\fi}
\def\textqed{\unskip\nobreak\penalty50
	\hskip2em\hbox{}\nobreak\hfil\qedsymbol
	\parfillskip=0pt \finalhyphendemerits=0}
\def\dispqed{\rlap{\qquad\qedsymbol}}

\opn\dis{dis}
\opn\height{height}
\opn\dist{dist}
\def\pnt{{\raise0.5mm\hbox{\large\bf.}}}

\opn\Lex{Lex}
\opn\conv{conv}
\opn\Ehr{Ehr}

\begin{document}
\title{The depth of a reflexive polytope}
	\author[T.~Hibi]{Takayuki Hibi}
\address[Takayuki Hibi]{Department of Pure and Applied Mathematics,
	Graduate School of Information Science and Technology,
	Osaka University,
	Suita, Osaka 565-0871, Japan}
\email{hibi@math.sci.osaka-u.ac.jp}
\author{Akiyoshi Tsuchiya}
\address[Akiyoshi Tsuchiya]{Department of Pure and Applied Mathematics,
	Graduate School of Information Science and Technology,
	Osaka University,
	Suita, Osaka 565-0871, Japan}
\email{a-tsuchiya@ist.osaka-u.ac.jp}
% \thanks{
% {\bf 2010 Mathematics Subject Classification:}
% Primary 13H10; Secondary 52B20. \\
% \,\, \, \, {\bf Keywords:}
% reflexive polytope, toric ring, depth, dual polytope, normal polytope, very ample polytope.
% }
\begin{abstract}
Given arbitrary integers $d$ and $r$ with $d \geq 4$ and $1 \leq r \leq d + 1$, a reflexive polytope $\Pc \subset \RR^d$ of dimension $d$ with $\depth K[\Pc] = r$ for which its dual polytope $\Pc^\vee$ is normal will be constructed, where $K[\Pc]$ is the toric ring of $\Pc$.
\end{abstract}

\maketitle
\thispagestyle{empty}

\section*{Introduction}
A {\em lattice polytope} is a convex polytope $\Pc \subset \RR^d$ each of whose vertices belongs to $\ZZ^d$.  We say that a lattice polytope $\Pc \subset \RR^d$ of dimension $d$ is {\em reflexive} if the origin ${\bf 0}$ of $\RR^d$ belongs to the interior $\Pc \setminus \partial \Pc$ of $\Pc$ and if the dual polytope
\[
\Pc^\vee = \{ \, \xb \in \RR^d \, : \, \langle \xb, \yb \rangle \leq 1, \, \forall \yb \in \Pc \, \}
\]
of $\Pc$ is a lattice polytope, where $\langle \xb, \yb \rangle$ is the canonical inner product of $\RR^d$.  Reflexive polytopes have been studied by many articles in the frame of algebraic combinatorics (e.g., \cite{Hibi}), algebraic geometry (e.g., \cite{mirror}) and Gr\"obner bases (e.g., \cite{perfect}).  Once the dimension is fixed, there exist only finitely many reflexive polytopes up to unimodular equivalence (\cite{Lag}).  Furthermore, in \cite{Kre}, a complete classification of the reflexive polytopes of dimension $\leq 4$ is achieved.
In particular, there are one reflexive polytope in dimension one, $16$ in dimension two, $4319$ in dimension three and 473800776 in dimension four.

Let $K[\tb, \tb^{-1}, s] = K[t_1, \ldots, t_d, t_1^{-1}, \ldots, t_d^{-1}, s]$ denote the Laurent polynomial ring in $d + 1$ variables over a field $K$.  Given a lattice polytope $\Pc \subset \RR^d$, the {\em toric ring} of $\Pc$ is the subring of $K[\tb, \tb^{-1}, s]$ which is generated by those Laurent monomials $t_1^{a_1}\cdots t_d^{a_d}s$ with $(a_1, \ldots, a_d) \in \Pc \cap \ZZ^d$.  One has {\rm Krull-dim}\,$K[\Pc] = \dim \Pc + 1$.  Let $\depth K[\Pc]$ denote the depth of the toric ring $K[\Pc]$.  
Please refer the reader to, e.g., \cite{BrunsHerzog} for the detailed information about depth.
The topic we are interested in  the behavior of $\depth K[\Pc]$ of a reflexive polytope $\Pc \subset \RR^d$.  When a lattice polytope $\Pc$ is reflexive with $\dim \Pc \leq 3$, its toric ring $K[\Pc]$ is (normal and) Cohen--Macaulay (e.g., \cite[Theorem 3.36]{unimodular}).  Hence $\depth K[\Pc] = \dim \Pc + 1$. 

In general, a lattice polytope $\Pc \subset \RR^d$ is called {\em normal} if for each $N = 1, 2, \ldots$ and for each $\ab \in N \Pc \cap \ZZ^d$, there exist $\ab_1, \ldots, \ab_N \in \Pc \cap \ZZ^d$ with $\ab = \ab_1 + \cdots + \ab_N$, where $N \Pc = \{ \, N \ab \, ; \, \ab \in \Pc \, \}$ is the $N$th dilation of $\Pc$.  If $\Pc$ is normal, then $K[\Pc]$ is normal.   

Now, our contribution to the study on reflexive polytopes is the following.

\begin{Theorem}
\label{Boston}
Given arbitrary integers $d$ and $r$ with $d \geq 4$ and $1 \leq r \leq d + 1$, there exists a reflexive polytope $\Pc \subset \RR^d$ of dimension $d$ with $\depth K[\Pc] = r$ for which its dual polytope $\Pc^\vee$ is normal.
\end{Theorem}
   
% For example, if $r = d + 1$, then the cube $C_d = [-1, 1]^d \subset \RR^d$ is reflexive with $\depth K[C_d] = d + 1$. 

Furthermore, we yield to the temptation to present the following.

\begin{Question}
\label{Sydney}
Given arbitrary integers $d, \, q$ and $r$ with $d \geq 4$ and $1 \leq q \leq r \leq d + 1$, does there exist a reflexive polytope $\Pc \subset \RR^d$ of dimension $d$ for which $\depth K[\Pc] = q$ and $\depth K[\Pc^\vee] = r$\,{\rm ?}
\end{Question}
The question is already open for $d=4$ and could be resolved in that dimension by computing the depth of all $4.10^8$ reflexive polytopes of that dimension.

A brief outline of the present paper is as follows.  First, a non-normal very ample reflexive polytope of dimension $4$ will be discovered (Example \ref{ex:non-nor}) and we will construct a non-normal very ample reflexive polytope for arbitrary dimension (Corollary \ref{Oberwolfach}).  As a result, it follows that a reflexive polytope $\Pc \subset \RR^d$ of dimension $d$ with ${\rm depth} K[\Pc]=1$ exists if $d \geq 4$ (Corollary \ref{depth1}). 
In Section \ref{sec:2}, for a proof of Theorem \ref{Boston}, we give indispensable examples (Example \ref{EX}) by computational experiments.
 Finally, in Section \ref{sec:3}, based on the computational results, a proof of Theorem \ref{Boston} will be given.  % At the present stage, it is premature to discuss Question \ref{Sydney}.  However, in Section $3$, computational experiments to attack Question \ref{Sydney} will be supplied.
In order to prove Theorem \ref{Boston}, we will focus on well-known higher-dimensional construction of reflexive polytopes, which is called {\em bipyramid construction}.
It is known that the bipyramid of a reflexive polytope $\Pc$ is normal if and only if $\Pc$ is normal.
In the present paper, we will also determine when the bipyramid of a reflexive polytope is very ample (Proposition \ref{KIX}). In addition, we can prove a similar result for lattice pyramid construction (Proposition \ref{KIX2}). 
% \newpage 
\subsection*{Acknowledgments}
We are very grateful to the anonymous referees for their insightful reports that led to significant improvements of the form of the paper. The second author was partially supported by Grant-in-Aid for JSPS Fellows 16J01549.

\section{Non-normal very ample reflexive polytopes}
We say that a lattice polytope $\Pc \subset \RR^d$ is {\em very ample} (\cite{Cox}) if for $N \gg 0$ and for each $\ab \in N \Pc \cap \ZZ^d$, there exist $\ab_1, \ldots, \ab_N \in \Pc \cap \ZZ^d$ with $\ab = \ab_1 + \cdots + \ab_N$.  Clearly every normal polytope is very ample.

\begin{Lemma}
	\label{ex:non-nor}
The reflexive polytope $\Pc \subset \RR^4$ of dimension $4$ which is the convex hull of the column vectors of 
	\begin{displaymath}
		\begin{bmatrix}
		0 & 1 & 0 & 0 &1&0&1&1&-1\\
		0&0&1&0&0&1&1&1&-1\\
		0&0&0&1&1&1&4&5&-3\\
		1&1&1&1&1&1&1&1&-2
	\end{bmatrix}
	\end{displaymath}
is non-normal very ample.
\end{Lemma}
\setcounter{MaxMatrixCols}{20}
\begin{proof}
The equations of the supporting hyperplanes which define the facets of $\Pc$ are 
\begin{eqnarray*}
x_4=1, 
& -3x_1+x_4=1, &
-3x_2+x_4=1,\\
-x_3+x_4=1,
& 4x_1+4x_2-x_3-3x_4=1, &
3x_1-2x_4=1,\\
3x_2-2x_4=1,
& -4x_1+x_3=1, &
-4x_2+x_3=1.
\end{eqnarray*}
Hence $\Pc^\vee$ is a lattice polytope.  Thus $\Pc$ is reflexive.

% \Hc_0=\{ (x_1,x_2,x_3,x_4) \in \RR^4 : x_4=1  \},
% \Hc_1=\{ (x_1,x_2,x_3,x_4) \in \RR^4 : -3x_1+x_4=1  \},
% \Hc_2=\{ (x_1,x_2,x_3,x_4) \in \RR^4 : -3x_2+x_4=1  \},
% \Hc_3=\{ (x_1,x_2,x_3,x_4) \in \RR^4 : -x_3+x_4=1  \},
% \Hc_4=\{ (x_1,x_2,x_3,x_4) \in \RR^4 : 4x_1+4x_2-x_3-3x_4=1  \},
% \Hc_5=\{ (x_1,x_2,x_3,x_4) \in \RR^4 : 3x_1-2x_4=1  \},
% \Hc_6=\{ (x_1,x_2,x_3,x_4) \in \RR^4 : 3x_2-2x_4=1  \},
% \Hc_7=\{ (x_1,x_2,x_3,x_4) \in \RR^4 : -4x_1+x_3=1  \},
% \Hc_8=\{ (x_1,x_2,x_3,x_4) \in \RR^4 : -4x_2+x_3=1  \}
% and for $0 \leq i \leq 8$ set $\Fc_i=\Pc \cap \Hc_i$.  Then $\Fc_0,\ldots,\Fc_8$ are the facets of $\Pc$.  Hence $\Pc$ is reflexive.
	
% Next, we show that $\Pc$ is not normal.  The lattice points in $\Pc$ are all the column vectors of 
One can easily see that $\Pc \cap \ZZ^4$ coincides with the set of column vectors of   
\begin{displaymath}
\begin{bmatrix}
0&1&0&0&1&0&1&1&-1&0&0&0\cr
0&0&1&0&0&1&1&1&-1&0&0&0 \cr
0&0&0&1&1&1&4&5&-3&0&1&-1 \cr
1&1&1&1&1&1&1&1&-2&0&0&0 \cr
\end{bmatrix}.
\end{displaymath}
Since 
\[
(1,1,3,2) = \frac{7(0,0,0,1) + 3(1,0,0,1) + 3(0,1,0,1) + 5(1,1,4,1) + 2(1,1,5,1)}{10},
\] 
it follows that $(1,1,3,2)$ belongs to $2 \Pc \cap \ZZ^4$.  However, $(1,1,3,2)$ cannot be written as the sum of two column vectors of the above matrix.  Thus $\Pc$ cannot be normal.

We now claim $\Pc$ is very ample.  Let $\Fc_0, \Fc_1, \ldots,\Fc_8$ denote the facets of $\Pc$, where $\Fc_0$ is the facet arising from the supporting hyperplane with the equation $x_4 = 1$.  Let $\Pc_i=\conv(\Fc_i \cup \{ {\bf 0}\})$  for $0 \leq i \leq 8$.  One has $N\Pc \cap \ZZ^4=\bigcup_{0 \leq i \leq 8}N\Pc_i \cap \ZZ^4$.  Let $N \geq 3$ and $\ab \in \ZZ^4$ belong to $N\Pc \cap \ZZ^4$.  A routine work says that each of $\Pc_1, \ldots, \Pc_8$ is normal.  Hence, if $\ab \in \bigcup_{1 \leq i \leq 8}N\Pc_i \cap \ZZ^4$, then there exist $\ab_1, \ldots, \ab_N \in \Pc \cap \ZZ^d$ with $\ab = \ab_1 + \cdots + \ab_N$.  Let $\ab \in N\Pc_0 \cap \ZZ^4$.  Since 
$
N\Pc_0 \cap \ZZ^4=\bigcup_{0 \leq k \leq N}k\Fc_0 \cap \ZZ^4,
$
there is $1 \leq k \leq N$ with $\ab \in k\Fc_0 \cap \ZZ^4$.  Now, it is shown \cite{Hig} that, for $3 \leq k \leq N$, there exist $\ab_1,\ldots,\ab_{k} \in \Fc_0 \cap \ZZ^4$ with $\ab=\ab_1+\cdots+\ab_{k}$ and that, for $k = 2$ with $\ab \neq (1,1,3,2)$, there exist $\ab_1, \ab_2 \in \Fc_0 \cap \ZZ^4$ with $\ab = \ab_1 + \ab_2$.  Since ${\bf 0} \in \Pc$, it follows that, if $\ab \neq (1,1,3,2)$, there exist $\ab_1, \ldots, \ab_N \in \Pc_0 \cap \ZZ^d$ with $\ab = \ab_1 + \cdots + \ab_N$. Finally,
\[
(1,1,3,2)=(1,1,4,1)+(0,0,0,1)+(0,0,-1,0)+\underbrace{{\bf 0}+\cdots+{\bf 0}}_{N-3}.
\]
As a result, for each $N \geq 3$ and for each $\ab \in N \Pc \cap \ZZ^d$, there exist $\ab_1, \ldots, \ab_N \in \Pc \cap \ZZ^d$ with $\ab = \ab_1 + \cdots + \ab_N$. Thus $\Pc$ is very ample, as desired.
\end{proof}

% \begin{Lemma}[{\cite[Theorem 2.4.7 (a)]{Cox}}]
%	\label{lem:product}
%	Let $\Pc_1$ and $\Pc_2$ be very ample polytopes.
%	Then $\Pc_1 \times \Pc_2$ is very ample.
% \end{Lemma}

\begin{Corollary}
\label{Oberwolfach}
A non-normal very ample reflexive polytope $\Pc \subset \RR^d$ of dimension $d$ exists if $d \geq 4$.
\end{Corollary}

\begin{proof}
Let $\Pc \subset \RR^4$ denote the lattice polytope of Lemma \ref{ex:non-nor}.  We claim 
\[
\Qc = \Pc \times [-1,1]^{d-4} \subset \RR^d
\] 
is a non-normal very ample reflexive polytope of dimension $d$.  Let $V \subset \ZZ^d$ be the set of vertices of the dual polytope $\Pc^\vee$.  Then the set of vertices of $\Qc^{\vee}$ is 
\[
\{ \, (\ab, 0, \ldots, 0) \in \RR^{d} \, : \, \ab \in V \, \} \cup \{\pm\eb_5,\ldots,\pm\eb_d \},
\]
where $\eb_1, \ldots, \eb_d$ are the canonical unit coordinate vectors of $\RR^d$.  In particular $\Qc^\vee$ is a lattice polytope.  Hence $\Qc \subset \RR^{d}$ is a reflexive polytope of dimension $d$.
  
Now, given an integer $N > 0$, one has
\[
N\Qc \cap \ZZ^d = \{ \, (\ab,a_5,\ldots,a_d) \in \ZZ^d \, : \, \ab \in N\Pc \cap \ZZ^4, -N \leq a_5,\ldots,a_d \leq N \, \}.
\]
It then follows that $\ab'=(\ab,x_5,\ldots,x_d) \in N\Qc \cap \ZZ^d$ is the sum of $N$ points belonging to $\Qc \cap \ZZ^d$ if and only if $\ab \in N\Pc \cap \ZZ^4$ is the sum of $N$ points belonging to $\Pc \cap \ZZ^4$.  As a result, since $\Pc \subset \RR^4$ is very ample, so is $\Qc \subset \RR^d$, as required.
\end{proof}

A crucial fact of the toric ring of a non-normal very ample polytope is

\begin{Lemma}[{\cite[Theorem 5.2]{Kat}}]
\label{germany}
Let $\Pc \subset \RR^d$ be a lattice polytope of dimension $d$ satisfying the condition
\begin{equation}
\label{span}
\ZZ^{d+1}= \sum_{\ab \in \Pc \cap \ZZ^d}\ZZ(\ab,1).
\end{equation}
	If $\Pc \subset \RR^d$ is a non-normal very ample polytope, then
	$\depth K[\Pc] = 1$.
\end{Lemma}

\begin{Corollary}
\label{depth1}
A reflexive polytope $\Pc \subset \RR^d$ of dimension $d$ with $\depth K[\Pc]=1$ exists if $d \geq 4$.
\end{Corollary}

\section{The depth of a reflexive polytope of dimension $4$}
\label{sec:2}
In this section, for a proof of Theorem \ref{Boston}, we give indispensable examples of reflexive polytope of dimension $4$.
According to the following steps, we compute the depth of reflexive polytopes of dimension $4$ by using existing software programs:

\begin{itemize}
	\item[(Step 1)] Take a reflexive polytope $\Pc$ of dimension $4$ in the list \cite{data};
	\item[(Step 2)] Determine the normality of $\Pc$ by using Normaliz \cite{Normaliz};
	\item[(Step 3)] If $\Pc$ is not normal, compute the depth of $K[\Pc]$ by using Macaulay2 \cite{M2};
	\item[(Step 4)] If $K[\Pc]$ is Cohen-Macaulay, i.e., ${\rm depth} K[\Pc]=5$, then determine the normality of the dual polytope $\Pc^{\vee}$ by using Normaliz again.
\end{itemize}

In general, computing the depth of a ring is hard.
In fact, in Step 3, the computation often stops before completing to compute the depth of $K[\Pc]$.
In the case, we restarted Step 1 for other reflexive polytope of dimension $4$.

First, we tried to compute the depth of all reflexive simplices of dimension $4$.
However, for several reflexive simplices, the program could not compute the depth and aborted its computation.
This happened for example for the $176$th reflexive simplex of dimension $4$ in the list \cite{data}, whose normalized volumes is $225$.
In general, for a lattice polytope with large volume or with many lattice points, the toric ring is complicated.
From these reasons we randomly took  a reflexive polytope with $5, 6,7,8$ or $9$ vertice in the list \cite{data} and computed the depth of its toric ring, and we continued these processes until we obtained the following three reflexive polytopes of dimension $4$ described in Example \ref{EX}.

\begin{Example}
\label{EX}
	{\rm
(a) The reflexive polytope $\Pc_1 \subset \RR^4$ of dimension $4$ which is the convex hull of the column vectors of 
\begin{displaymath}
	\begin{bmatrix}
	1&0&0&0 &0&1&1&-2&-3\\
	0&1&0&0&-1&0&-1&-1&-1\\
	0&0&1&0&-1&1&-1&0&0\\
	0&0&0&1&0&0&0&-1&-1
\end{bmatrix}
	\end{displaymath}
satisfies $\depth K[\Pc_1] = 2$.
On the other hand, the dual polytope $\Pc_1^{\vee}$ is normal.
\smallskip
	
	(b) The reflexive polytope $\Pc_2 \subset \RR^4$ of dimension $4$ which is the convex hull of the column vectors of 
	\begin{displaymath}
	\begin{bmatrix}
	1&0&0&0 &-3&3&\\
	0&1&0&0&-2&-1&\\
	0&0&1&0&0&-3&\\
	0&0&0&1&0&-1&
	\end{bmatrix}
	\end{displaymath}
satisfies $\depth K[\Pc_2] = 3$.
On the other hand, the dual polytope $\Pc_2^{\vee}$ is normal.
\smallskip
	
	(c) The reflexive polytope $\Pc_3 \subset \RR^4$ of dimension $4$ which is the convex hull of the column vectors of 
	\begin{displaymath}
	\begin{bmatrix}
	1&0&0&0 &-1&-2&\\
	0&1&0&0&-1&-1&\\
	0&0&1&0&-1&2&\\
	0&0&0&1&0&-2&
	\end{bmatrix}
	\end{displaymath}
satisfies $\depth K[\Pc_3] = 4$.
On the other hand, the dual polytope $\Pc_3^{\vee}$ is normal.}
\end{Example}

\section{Proof of Theorem \ref{Boston}}
\label{sec:3}
In this section, we give a proof of Theorem \ref{Boston}.
Recall that, in general, the {\em bipyramid} of a convex polytope $\Pc \subset \RR^d$ is the convex polytope $\bipyr(\Pc) \subset \RR^{d+1}$ which is the convex hull of 
\[
\{ \, (\ab, 0) \in \RR^{d+1} \, : \, \ab \in \Pc \, \} \cup \{ (0, \ldots, 0, 1), (0, \ldots, 0, - 1) \} \subset \RR^{d+1}. 
\]

\begin{Lemma}
\label{bipyr}
	Let $\Pc \subset \RR^d$ be a reflexive polytope of dimension $d$.
	Then $\bipyr(\Pc) \subset \RR^{d+1}$ is a reflexive polytope of dimension $d + 1$ with
	\begin{eqnarray}
	\label{depth}
	\depth K[\bipyr(\Pc)] = \depth K[\Pc] + 1.
	\end{eqnarray}
\end{Lemma}

\begin{proof}
If $V \subset \ZZ^d$ is the set of vertices of the dual polytope $\Pc^\vee$, then that of $(\bipyr(\Pc))^\vee$ is $\{ \, (\ab, \pm 1) \in \RR^{d+1} \, : \, \ab \in V \, \}$.  In particular $(\bipyr(\Pc))^\vee$ is a lattice polytope.  Hence $\bipyr(\Pc) \subset \RR^{d+1}$ is a reflexive polytope of dimension $d + 1$.  Furthermore, since 
\[
K[\bipyr(\Pc)] \cong (K[\Pc])[\,y, z\,]/(yz - s^2),
\]
the required formula (\ref{depth}) follows.
\end{proof}

The bipyramid $\bipyr(\Pc)$ of a reflexive polytope $\Pc \subset \RR^d$ of dimension $d$ is normal if and only if $\Pc$ is normal.  However, Lemma \ref{bipyr} says that the bipyramid of a non-normal very ample reflexive polytope cannot be very ample.  In fact, 
 
\begin{Proposition}
\label{KIX}
	Let $\Pc \subset \RR^d$ be a reflexive polytope of dimension $d$. % satisfying the condition (\ref{span}).
	Then the following conditions are equivalent{\em :}
	\begin{enumerate}
		\item[{\rm (i)}] $\Pc$ is normal{\em ;}
		\item[{\rm (ii)}] $\bipyr(\Pc)$ is normal{\em ;}
		\item[{\rm (iii)}] $\bipyr(\Pc)$ is very ample. 
	\end{enumerate}
\end{Proposition}

\begin{proof}
One has (i) $\Leftrightarrow$ (ii) $\Rightarrow$ (iii).  Suppose that $\bipyr(\Pc)$ is very ample. 
Then $\bipyr(\Pc)$ satisfies the condition (\ref{span}).
 Lemma \ref{bipyr} says that $\depth K[\bipyr(\Pc)] > 1$.  It then follows from Lemma \ref{germany} that $\bipyr(\Pc)$ must be normal.
Hence (iii) $\Rightarrow$ (ii) follows.
\end{proof}

Recall that the {\em lattice pyramid} of a lattice polytope $\Pc \subset \RR^d$ is the convex polytope
${\rm pyr(\Pc)}$ which is the convex hull of 
\[
\{ \, (\ab, 0) \in \RR^{d+1} \, : \, \ab \in \Pc \, \} \cup \{ (0, \ldots, 0, 1) \} \subset \RR^{d+1}. 
\]
By the same proof of Proposition \ref{KIX}, we can show the following.

\begin{Proposition}
	\label{KIX2}
	Let $\Pc \subset \RR^d$ be a lattice polytope of dimension $d$. % satisfying the condition (\ref{span}).
	Then the following conditions are equivalent{\em :}
	\begin{enumerate}
		\item[{\rm (i)}] $\Pc$ is normal{\em ;}
		\item[{\rm (ii)}] ${\rm pyr}(\Pc)$ is normal{\em ;}
		\item[{\rm (iii)}] ${\rm pyr}(\Pc)$ is very ample. 
	\end{enumerate}
\end{Proposition}

Before proving Theorem \ref{Boston}, we show the following lemma.
\begin{Lemma}
\label{DUAL}
Let $\Pc \subset \RR^d$ be a reflexive polytope of dimension $d$ and suppose that its dual polytope $\Pc^\vee \subset \RR^d$ is normal.  Then each of the dual polytopes $(\Pc \times [-1,1])^\vee$ and $(\bipyr(\Pc))^\vee$ is normal.
\end{Lemma}

\begin{proof}
One has $(\Pc \times [-1,1])^\vee = \bipyr(\Pc^\vee)$ and $(\bipyr(\Pc))^\vee = \Pc^\vee \times [-1,1]$.  Thus the desired result follows immediately.
\end{proof}

We are now in the position to give a proof of Theorem \ref{Boston}.
  
\begin{proof}[Proof of Theorem \ref{Boston}]
Let $d = 4$.  It follows from a routine computation that the dual polytope of the reflexive polytope of Lemma \ref{ex:non-nor} as well as the dual polytope of each of the reflexive polytopes (a), (b) and (c) of Example \ref{EX} is normal.  Corollary \ref{depth1} with $d = 4$ together with Example \ref{EX} then guarantees the existence of a reflexive polytope $\Pc \subset \RR^4$ of dimension $4$ with $\depth K[\Pc] = r$ for which $\Pc^\vee$ is normal, where $1 \leq r \leq 4$.  Furthermore, the cube $C_4 = [-1,1]^{4} \subset \RR^4$ is a reflexive polytope of dimension $4$ with $\depth K[C_4] = 5$ and $C_4^\vee$ is normal.  Thus a proof for $d = 4$ is done.  

Let $d \geq 5$.  The proof of Corollary \ref{Oberwolfach} together with Lemma \ref{DUAL} yields a non-normal very ample reflexive polytope $\Pc \subset \RR^d$ of dimension $d$ for which $\Pc^\vee$ is normal and $\depth K[\Pc] = 1$ (Corollary \ref{depth1}).  Now, suppose that, for each $1 \leq r \leq d$, there is a reflexive polytope $\Pc \subset \RR^{d-1}$ of dimension $d - 1$ with $\depth K[\Pc] = r$ for which $\Pc^\vee$ is normal.  Lemma \ref{bipyr} together with Lemma \ref{DUAL} then yields a reflexive polytope $\Pc \subset \RR^{d}$ of dimension $d$ with $\depth K[\Pc] = r$ for each $2 \leq r \leq d + 1$ for which $\Pc^\vee$ is normal.  This completes a proof of Theorem \ref{Boston}.  
\end{proof}

\end{document}